\documentclass{amsproc}
\usepackage{amsfonts}

\theoremstyle{plain}

\newtheorem{theorem}{Theorem}
\numberwithin{equation}{section}

\begin{document}
\title{A short proof of Combinatorial Nullstellensatz}
\author{Mateusz Micha\l ek}
\address{Mathematical Institute of the Polish Academy of Sciences, ul. \'{S}%
niadeckich 8, 00-950 Warszawa, Poland}
\email{wajcha2@poczta.onet.pl}
\address{Unit\'{e} Mixte de Recherche 5582 CNRS --- Universit\'{e} Grenoble
1 \\
100 rue des Maths, BP 74, 38402 St Martin d'H\`{e}res, France}
\maketitle

In this note we give a short, direct proof of the following theorem of Alon 
\cite{AlonCN}.

\begin{theorem}
\label{CombNull}\emph{(Combinatorial Nullstellensatz \cite{AlonCN})} Let $%
\mathbb{F}$ be an arbitrary field, and let $P(x_{1},\ldots ,x_{n})$ be a
polynomial in $\mathbb{F}[x_{1},\ldots ,x_{n}]$. Suppose the degree $\deg (P)
$ of $P$ is $\sum_{i=1}^{n}k_{i}$, where each $k_{i}$ is a non-negative
integer, and suppose the coefficient of $x_{1}^{k_{1}}x_{2}^{k_{2}}\cdots
x_{n}^{k_{n}}$ in $P$ is non-zero. Then for any subsets $A_{1},\ldots ,A_{n}$
of $\mathbb{F}$ satisfying $\left\vert A_{i}\right\vert \geq k_{i}+1$ for
all $i=1,2,\ldots ,n$, there are $a_{1}\in A_{1},\ldots ,a_{n}\in A_{n}$ so
that $P(a_{1},\ldots ,a_{n})\neq 0$.
\end{theorem}

\begin{proof}
The proof is by induction on the degree of $P$. If $\deg (P)=1$ the theorem is trivial. Suppose $\deg (P)>1$ and $P$ satisfies the
assumptions of the theorem but the assertion is false, that is $P(x)=0$ for
every $x\in A_{1}\times \ldots \times A_{n}$. Without loss of generality $k_1>0$. Fix $a\in A_{1}$ and write%
\begin{equation}
P=(x_{1}-a)Q+R  \tag{*}
\end{equation}%
using the usual long division algorithm for polynomials. Formally equation
(*) is an identity in the ring of polynomials in one variable $x_{1}$ with
coefficients in the ring $\mathbb{F}[x_{2},\ldots ,x_{n}]$. Since the degree of $%
R$ in the variable $x_{1}$ is strictly less than $\deg (x_{1}-a)$, the
polynomial $R$ does not contain $x_{1}$ at all. By the assumption on $P$ it
follows that $Q$ must have a nonvanishing monomial of the form $%
x_{1}^{k_{1}-1}x_{2}^{k_{2}}\cdots x_{n}^{k_{n}}$ and $\deg
(Q)=\sum_{i=1}^{n}k_{i}-1=\deg (P)-1$.

Take any $x\in \{a\}\times A_{2}\times \ldots \times A_{n}$ and substitute
to the equation (*). Since $P(x)=0$ it follows that $R(x)=0$, too. But $R$
does not contain $x_{1}$, hence $R$ vanishes also on $(A_{1}-\{a\})\times
A_{2}\times \ldots \times A_{n}$. Now substitute any $x\in
(A_{1}-\{a\})\times A_{2}\times \ldots \times A_{n}$ to (*). Since $x_{1}-a$
is non-zero, it follows that $Q(x)=0$. So $Q$ vanishes on $%
(A_{1}-\{a\})\times A_{2}\times \ldots \times A_{n}$, which contradicts the
inductive assumption.
\end{proof}

The theorem is also true when $\mathbb{F}$ is any integral domain. If we
omit the integrality assumption and take any commutative ring with unity as $%
\mathbb{F}$, then the proof works if for any two elements $a,b\in A_{i}$, ($%
i=1,\ldots ,n)$, the difference $a-b$ is not a zero divisor in $\mathbb{F}$.

Theorem \ref{CombNull} has many useful applications in Combinatorics, Graph
Theory, and Additive Number Theory (see \cite{AlonCN}). As an illustration
we apply it in the proof of Erd\"os Heilbronn Conjecture. It was first proved by da Silva and
Hamidoune, and then by Alon, Nathanson and Ruzsa in a more general setting.

\begin{theorem}
Let $p$ be a prime and let $A$ and $B$ be two non-empty subsets of $\mathbb{Z%
}_{p}$. Let 
\begin{equation*}
C=\{x\in \mathbb{Z}_{p}:x=a+b\text{ for some }a\in A,b\in B,a\neq b\}.
\end{equation*}%
Then 
\begin{equation*}
\left\vert C\right\vert \geq \min (p,\left\vert A\right\vert +\left\vert
B\right\vert -3).
\end{equation*}
\end{theorem}

\begin{proof}
Only the case $p>2$ is interesting. If $\min (p,\left\vert A\right\vert
+\left\vert B\right\vert -3)=p$, then for any $g\in \mathbb{Z}_{p}$ the sets 
$A$ and $g-B$ have got at least two different elements in common. Let $a$ be
one that is different from $\frac{g}{2}$, then $g=a+b$ for some $b\in B$
different from $a$. This proves that $g\in C$, so $C=\mathbb{Z}_{p}$.
Suppose now that $\left\vert A\right\vert +\left\vert B\right\vert -3<p$ and
the theorem does not hold. In that case there exists a set $D$ such that $%
C\subseteq D$ and $\left\vert D\right\vert =\left\vert A\right\vert
+\left\vert B\right\vert -4$. We define two polynomials: 
\begin{equation*}
P(x,y)=\prod_{d\in D}(x+y-d)\quad \text{and}\quad Q(x,y)=P(x,y)(x-y).
\end{equation*}%
Clearly $P(a,b)=0$ for any $a\in A$, $b\in B$, $a\neq b$, so $Q(a,b)=0$ for
any $a\in A$, $b\in B$. One can also see that the coefficient of $x^{i}y^{j}$
in $P(x,y)$ is equal to $\binom{\left\vert D\right\vert }{{i}}$ if $%
i+j=\left\vert D\right\vert $. So if $i+j=\left\vert D\right\vert +1$ then
the coefficient of $x^{i}y^{j}$ in $Q(x,y)$ is equal to ${\binom{\left\vert
D\right\vert }{{i-1}}}-{\binom{\left\vert D\right\vert }{{i}}}$. This is
equal to zero if and only if $i=\frac{|D|+1}{2}$ in $\mathbb{Z}_{p}$. Since $%
\left\vert D\right\vert +1=\left\vert A\right\vert +\left\vert B\right\vert
-3$ one of the coefficients of either $x^{\left\vert A\right\vert
-1}y^{\left\vert B\right\vert -2}$ or $x^{\left\vert A\right\vert
-2}y^{\left\vert B\right\vert -1}$ is nonzero. Using Theorem \ref{CombNull}
and the fact that $\deg Q=|A|+|B|-3$ we get a contradiction.
\end{proof}
\section*{Acknowledgements}
I would like to thank Jarek Grytczuk for encouraging me to publish this paper.

\end{document}